\documentclass{birkjour}
\usepackage{epigraph}
\newtheorem{theorem}{Theorem}[section]
\newtheorem{lem}[theorem]{Lemma}

\theoremstyle{definition}

\theoremstyle{remark}
\newtheorem{remark}[theorem]{Remark}

\numberwithin{equation}{section}

\begin{document}

\title{A $q$-analogue for Euler's $\zeta(6)=\pi^6/945$}

%    Remove any unused author tags.

%    author one information

\author{Ankush Goswami}
\address{Department of Mathematics, University of Florida, 
Gainesville, Fl 32603}
%\curraddr{}
\email{ankush04@ufl.edu}
%\thanks{This work was completed}

%    author two information
%\author{}
%\address{}
%\curraddr{}
%\email{}
%\thanks{}


\date{}

\dedicatory{In honour of Prof. George Andrews on his $80^{th}$ birthday}
\begin{abstract}
Recently, Z.-W Sun \cite{Sun} obtained $q$-analogues of Euler's formula for $\zeta(2)$ and $\zeta(4)$. Sun's formula were based on identities satisfied by triangular numbers and properties of Euler's $q$-Gamma function. In this paper, we obtain a $q$-analogue of $\zeta(6)=\pi^6/945$. Our main results are stated in Theorems 2.1 and 2.2 below.  
\end{abstract}
\maketitle
\section{Introduction}
Recently, Sun \cite{Sun} obtained a very nice $q$-analogue of Euler's formula $\zeta(2)=\pi^2/6$.  Motivated by this, the author obtained the $q$-analogue of $\zeta(4)=\pi^4/90$ and noted that it was simultaneously and independently obtained by Sun. %This led to the natural question : Can one find the $q$-analogue of $\zeta(6)=\pi^6/945$? In his very recent version \cite{Sun}, 
Further, Sun commented that one does not know how to find a $q$-analogue of Euler's formula for $\zeta(6)$ and beyond. This further motivated the author to consider the problem and indeed we obtained the $q$-analogue of $\zeta(6)$. As we shall see shortly, the $q$-analogue formulation of $\zeta(6)$ is more difficult as compared to $\zeta(2)$ and $\zeta(4)$ due to an extra term that shows up in the identity; however in the limit as $q\rightarrow 1^{-}$, this term $\rightarrow 0$. We also state the $q$-analogue of $\zeta(4)=\pi^4/90$ since we found it independent of Sun's result, however we skip the proof of this since it essentially uses the same idea as Sun.
\section{Main theorems}
\begin{theorem}
For a complex $q$ with $|q|<1$ we have
\begin{eqnarray}\label{qaz4}
\sum_{k=0}^\infty \dfrac{q^{2k}\;P_2(q^{2k+1})}{(1-q^{2k+1})^4}=\prod_{n=1}^\infty\dfrac{(1-q^{2n})^8}{(1-q^{2n-1})^8}
\end{eqnarray}
where $P_2(x) = x^2+4x+1$. In other words, $(\ref{qaz4})$ gives a $q$-analogue of $\zeta(4)=\pi^4/90$.
\end{theorem} 
\begin{theorem}
For a complex $q$ with $|q|<1$ we have
\begin{eqnarray}\label{qaz6}
\sum_{k=0}^\infty \dfrac{q^{k}(1+q^{2k+1})\;P_4(q^{2k+1})}{(1-q^{2k+1})^6}-\phi^{12}(q)=256q\prod_{n=1}^\infty\dfrac{(1-q^{2n})^{12}}{(1-q^{2n-1})^{12}}
\end{eqnarray}
where $P_4(x) = x^4+236x^3+1446x^2+236x+1$ and $\phi(q)=\displaystyle\prod_{n=1}^\infty (1-q^n)$ is the Euler's function. In other words, $(\ref{qaz6})$ gives a $q$-analogue of $\zeta(6)=\pi^6/945$.
\end{theorem}
\begin{remark}
We note that $\phi^{12}(q)$ has a beautiful product representation and from (\ref{qaz6}), it is uniquely determined by 
\begin{eqnarray}\label{cusp}
\phi^{12}(q)=\sum_{k=0}^\infty \dfrac{q^{k}(1+q^{2k+1})\;P_4(q^{2k+1})}{(1-q^{2k+1})^6}-256q\prod_{n=1}^\infty\dfrac{(1-q^{2n})^{12}}{(1-q^{2n-1})^{12}}.
\end{eqnarray} 
In the general $q$-analogue formulation (see \cite{Gos}), we do not have very elegant representations of these functions, although we obtain expressions for them similar to (\ref{cusp}). 
\end{remark}
\begin{remark}
Since the coefficients in the $q$-series expansion of $\phi^{12}(q)$ are related to the pentagonal numbers by Euler's pentagonal number theorem and that for the product in the right-hand side of (\ref{qaz6}) are related to the triangular numbers, it will be worthwhile to understand the relationships of these coefficients via identity (\ref{qaz6}).
\end{remark}
\section{Some useful lemmas}
Let $q=e^{2\pi i\tau}$, $\tau\in\mathcal{H}$ where $\mathcal{H}=\{\tau\in\mathbb{C} : \mbox{Im}(\tau)>0$\}. Then the Dedekind $\eta$-function defined below is a modular form of weight $\frac{1}{2}$  
\begin{eqnarray}\label{D_eta}
\eta(\tau) = q^{1/24}\prod_{n=1}^\infty (1-q^n).
\end{eqnarray}
Also, let us denote by $\psi(q)$ the following sum.
\begin{eqnarray}
\psi(q)=\sum_{n=0}^\infty q^{T_n}
\end{eqnarray}
where $T_n=\dfrac{n(n+1)}{2}$ (for $n=0,1,2,..$) are triangular numbers. Then we have the following well-known result due to Gauss 
\begin{lem}\label{Gauss}
\begin{eqnarray}
\psi(q)=\prod_{n=1}^\infty \dfrac{(1-q^{2n})}{(1-q^{2n-1})}.
\end{eqnarray}
\end{lem}
Thus we have from Lemma $(\ref{Gauss})$ that 
\begin{eqnarray}
%&&\prod_{n=1}^\infty\dfrac{(1-q^{2n})^8}{(1-q^{2n-1})^8} = \psi^8(q) = \sum_{n=1}^\infty t_8(n) q^n\\
%&\mbox{and,}&\nonumber\\
\prod_{n=1}^\infty\dfrac{(1-q^{2n})^{12}}{(1-q^{2n-1})^{12}} = \psi^{12}(q) = \sum_{n=1}^\infty t_{12}(n) q^n
\end{eqnarray} 
where $t_{12}(n)$ is the number of ways of representing a positive integer $n$ as a sum of 12 triangular numbers. Next, we have the following well-known result of Ono, Robins and Wahl \cite{ORW}.
%\begin{theorem}[Ono, Robins, Wahl]
%For a positive integer $n$, we have
%\begin{eqnarray}
%t_8(n) = \sigma^\#_3(n+1)
%\end{eqnarray}
%where
%\begin{eqnarray}
%\sigma^\#_3(n) = \sum_{\substack{d|n \\ n/d \;\equiv 1\;(mod\;2)}}d^3
%\end{eqnarray} 
%\end{theorem}
\begin{theorem}
Let, $\eta^{12}(2\tau)=\displaystyle\sum_{k=0}^\infty a(2k+1)q^{2k+1}$ then for a positive integer $n$ we have
\begin{eqnarray}
t_{12}(n) = \dfrac{\sigma_5(2n+3)-a(2n+3)}{256}
\end{eqnarray}
where
\begin{eqnarray}
\sigma_5(n) = \sum_{d|n}d^5.
\end{eqnarray} 
\end{theorem}
\section{Proof of Theorem 2.2}
Since $\zeta(6)=\dfrac{\pi^6}{945}$ has the following equivalent form
%\begin{eqnarray}
%\sum_{k=0}^\infty \dfrac{1}{(2k+1)^4} = \dfrac{15}{16}\zeta(4) = \dfrac{\pi^4}{96}
%\end{eqnarray}
%and,
\begin{eqnarray}
\sum_{k=0}^\infty \dfrac{1}{(2k+1)^6} = \dfrac{63}{64}\zeta(6) = \dfrac{\pi^6}{960},
\end{eqnarray}
it will be sufficient to get the $q$-analogue of (4.1). Now, from $q$-analogue of Euler's Gamma function we know that 
\begin{eqnarray}\label{EG}
\lim_{q\uparrow 1} \;(1-q)\prod_{n=1}^\infty\dfrac{(1-q^{2n})^2}{(1-q^{2n-1})^2}=\dfrac{\pi}{2}
\end{eqnarray}
so that from $(\ref{EG})$ we have
%\begin{eqnarray}
%\lim_{q\uparrow 1} \;(1-q)^4\prod_{n=1}^\infty\dfrac{(1-q^{2n})^8}{(1-q^{2n-1})^8}=\dfrac{\pi^4}{16}
%\end{eqnarray} 
%and,
\begin{eqnarray}
\lim_{q\uparrow 1} \;(1-q)^6\prod_{n=1}^\infty\dfrac{(1-q^{2n})^{12}}{(1-q^{2n-1})^{12}}=\dfrac{\pi^6}{64}.
\end{eqnarray} 
Next, we consider the following infinite series 
%\begin{eqnarray}
%S_4(q) := \sum_{k=0}^\infty \dfrac{q^{2k}\;P_2(q^{2k+1})}{(1-q^{2k+1})^4}
%\end{eqnarray}
%and,
\begin{eqnarray}
S_6(q) := \sum_{k=0}^\infty \dfrac{q^{k}(1+q^{2k+1})\;P_4(q^{2k+1})}{(1-q^{2k+1})^6}
\end{eqnarray}
where $P_4(x)=x^4+236x^3+1446x^2+236x+1$.

By partial fractions we have 
%\begin{eqnarray}\label{P4}
%S_4(q) = \sum_{k=0}^\infty q^{2k}\left\{\dfrac{6}{(1-q^{2k+1})^4}-\dfrac{6}{(1-q^{2k+1})^3}+\dfrac{1}{(1-q^{2k+1})^2}\right\}
%\end{eqnarray}
%and,
\begin{eqnarray}\label{P6}
S_6(q) = \sum_{k=0}^\infty q^k\left\{\dfrac{3840}{(1-q^{2k+1})^6}-\dfrac{9600}{(1-q^{2k+1})^5}+\dfrac{8160}{(1-q^{2k+1})^4}\right.\nonumber\\\left.-\dfrac{2640}{(1-q^{2k+1})^3}+\dfrac{242}{(1-q^{2k+1})^2}-\dfrac{1}{(1-q^{2k+1})}\right\}.
\end{eqnarray}
\begin{lem}
With $S_6(q)$ represented by $(\ref{P6})$ we have
%\begin{eqnarray}\label{S4}
%S_4(q) = \sum_{n=0}^\infty t_8(n) q^n
%\end{eqnarray}
%and
\begin{eqnarray}\label{S6}
S_6(q) = 256q\sum_{n=0}^\infty t_{12}(n)q^n + \phi^{12}(q).
\end{eqnarray}
\end{lem} 

%From $(\ref{P4})$ we have
%\begin{eqnarray*}
%S_4(q) &=& \sum_{k=0}^\infty\sum_{j=0}^\infty q^{2k}\left\{6\binom{-4}{j}-6\binom{-3}{j}+\binom{-2}{j}\right\}(-q)^{j(2k+1)}\\
%&=&\sum_{k=0}^\infty\sum_{j=0}^\infty \left((j+1)(j+2)(j+3)-3(j+1)(j+2)+(j+1)\right) q^{2k+j(2k+1)}\\
%&=&\sum_{k=0}^\infty\sum_{j=0}^\infty (j+1)^3 q^{2k+j(2k+1)}\\
%&=& \sum_{k=0}^\infty\sum_{j=0}^\infty (j+1)^3 q^{(j+1)(2k+1)-1}\\
%&=& \sum_{n=0}^\infty\left(\sum_{\substack{d|n+1\\(n+1)/d\;\equiv 1\;(mod\;2)}}d^3\right) q^n\\
%&=& \sum_{n=0}^\infty\sigma^\#_3(n+1)q^n
%\end{eqnarray*}
%and so $(\ref{S4})$ follows from theorem (2.2).
\begin{proof}
From $(\ref{P6})$ we have 
\begin{eqnarray}
S_6(q) &=& \sum_{k=0}^\infty\sum_{j=0}^\infty q^{k}\left\{3840\binom{-6}{j}-9600\binom{-5}{j}+8160\binom{-4}{j}\right.\nonumber\\&&\left.-2640\binom{-3}{j}+242\binom{-2}{j}-\binom{-1}{j}\right\}(-q)^{j(2k+1)}\nonumber\\
&=& \sum_{k=0}^\infty\sum_{j=0}^\infty \left\{32(j+1)(j+2)(j+3)(j+4)(j+5)\right.\nonumber\\&&\left.-400(j+1)(j+2)(j+3)(j+4)+1360(j+1)(j+2)(j+3)\right.\nonumber\\&&\left.-1320(j+1)(j+2)+242(j+1)-1\right\}q^{k+j(2k+1)}\nonumber\\
&=& \sum_{k=0}^\infty\sum_{j=0}^\infty (2j+1)^5 q^{\frac{(2j+1)(2k+1)-1}{2}}\nonumber\\
&=& \sum_{n=0}^\infty \sigma_5(2n+1)q^n\nonumber\\
&=& 1 + \sum_{n=1}^\infty \sigma_5(2n+1)q^n\nonumber\\
&=& 1 + q\sum_{n=0}^\infty \sigma_5(2n+3)q^{n}.\nonumber 
\end{eqnarray}
Also from $(\ref{D_eta})$ we have
\begin{eqnarray}
\phi^{12}(q) &=& \dfrac{\eta^{12}(\tau)}{q^{\frac{1}{2}}}\nonumber\\
&=& \sum_{n=0}^\infty a(2n+1)q^n\nonumber\\
&=& 1+\sum_{n=1}^\infty a(2n+1)q^n\nonumber\\
&=& 1+q\sum_{n=0}^\infty a(2n+3)q^n.\nonumber
\end{eqnarray} 
Thus from above we have
\begin{eqnarray*}
S_6(q)-\phi^{12}(q) &=& q\sum_{n=0}^\infty \left\{\sigma_5(2n+3)-a(2n+3)\right\}q^n\\
&=& 256\;q\sum_{n=0}^\infty t_{12}(n) q^n
\end{eqnarray*}
where the last step follows from Theorem 3.2. This completes the proof of Theorem 2.2.
\end{proof}\mbox{}\\
We also note that 
%\begin{eqnarray}
%\lim_{q\uparrow 1}\;(1-q)^4 S_4(q) = \sum_{k=0}^\infty \dfrac{6}{(2k+1)^4}
%\end{eqnarray}
%and 
\begin{eqnarray}
\lim_{q\uparrow 1}\;(1-q)^6 (S_6(q)-\phi^{12}(q)) &=& \lim_{q\uparrow 1}\;(1-q)^6 S_6(q)-\lim_{q\uparrow 1}\;(1-q)^6\phi^{12}(q)\nonumber\\
&=&\sum_{k=0}^\infty \dfrac{3840}{(2k+1)^6}
\end{eqnarray}
where $\lim_{q\uparrow 1}\;(1-q)^6\phi^{12}(q)=0$ and $q\uparrow 1$ indicates $q\rightarrow 1$ from within the unit disk. Hence, combining equations (4.1), (4.3), (4.7) and Lemma 4.1, Theorem 2.2 follows. 
\section{Acknowledgement}
I am grateful to Prof. Krishnaswami Alladi for carrying out discussions pertaining to the function $\phi(q)$ and for his encouragement. I also thank Prof. George Andrews for going through my proof and providing me a few useful references. 
\end{document}